\documentclass[a4paper,12pt]{article}

\usepackage{amssymb,latexsym,amsmath,enumerate,verbatim,amsfonts,amsthm,amsbsy,amscd}
\usepackage[active]{srcltx}

\newtheorem{theorem}{Theorem}[section]
\newtheorem{lemma}[theorem]{Lemma}
\newtheorem{corollary}[theorem]{Corollary}

\theoremstyle{definition}
\newtheorem{remark}{Remark}[section]
\newtheorem{example}{Example}[section]

\newtheorem{definition}{Definition}[section]

\makeatletter
\def\th@plain{%
  \thm@notefont{}
  \itshape 
}
\def\th@definition{%
  \thm@notefont{}
  \normalfont 
}
\makeatother

\DeclareMathOperator{\co}{co}

\DeclareMathOperator{\dom}{dom}

\DeclareMathOperator*{\argmin}{argmin}

\begin{document}

\title{{\textbf{SDP Duals without Duality Gaps for a Class of Convex Minimax Programs}\footnote{Research was partially supported by a grant from the Australian Research Council.}}}

\author{\textsc{V. Jeyakumar}\thanks{Corresponding author. Department of Applied Mathematics, University of New South Wales, Sydney 2052, Australia. E-mail: v.jeyakumar@unsw.edu.au} \qquad and \qquad \textsc{J. Vicente-P\'erez}\thanks{Department of Applied Mathematics, University of New South Wales, Sydney 2052, Australia. This author has been partially supported by the MICINN of Spain, Grant MTM2011-29064-C03-02. E-mail: jose.vicente@ua.es} }

\date{May 5, 2013}

\maketitle

\begin{abstract}
\noindent
In this paper we introduce a new dual program, which is representable as a semi-definite linear programming problem, for a primal convex minimax programming model problem and show that there is no duality gap between the primal and the dual whenever the functions involved are SOS-convex polynomials. Under a suitable constraint qualification, we derive strong duality results for this class of minimax problems. Consequently, we  present applications of our results to robust SOS-convex programming problems under data uncertainty and to minimax fractional programming problems with SOS-convex polynomials.  We obtain these results by first establishing sum of squares polynomial representations of non-negativity of a convex max function over a system of SOS-convex constraints. The new class of SOS-convex polynomials is an important subclass of convex polynomials and it includes convex quadratic functions and separable convex polynomials. The SOS-convexity of polynomials can numerically be checked by solving semi-definite programming problems whereas numerically verifying convexity of polynomials is generally very hard.

\bigskip

\end{abstract}


\section{Introduction}
\label{SEC1}


Consider the minimax programming problem
\begin{equation*}
\begin{array}{ccl}
(P)   &   \inf\limits_{x\in\mathbb{R}^n}  &  \max\limits_{j\in \mathbb{N}_r}\,p_j(x)\\
	       &  \text{s.t.}  &  g_i(x) \leq 0, \ i\in \mathbb{N}_m,
\end{array}
\end{equation*}
where $p_j$, for $j\in \mathbb{N}_r:=\{1,\ldots,r\}$, and $g_i$, for $i\in \mathbb{N}_m:=\{1,\ldots,m\}$, are real polynomials on $\mathbb{R}^n$.
\smallskip

Discrete minimax model problems of the form $(P)$ arise in many areas of applications in engineering and commerce as resource allocation and planning problems (\cite{ibaraki} and other references therein). More recently, these models have appeared in robust optimization \cite{robust,bert} which is becoming increasingly important in optimization due to the reality of uncertainty in many real-world optimization problems and the importance of finding solutions that are immunized against data uncertainty. For instance, consider the optimization model problem with the data uncertainty in the constraints and in the objective function:
\begin{equation*}
	\inf \{f_0(x,v_0) \ : \ f_i(x,v_i) \leq 0,\ \forall i=1,\ldots,k\},
\end{equation*}
where $v_i\in\mathbb{R}^{n_i}$ is an uncertain parameter belonging to a finite uncertainty set $\mathcal V_i:=\{ v_i^1,\ldots,v_i^{s_i}\}$ for each $i\in\{0\}\cup\mathbb{N}_k$. The robust counterpart of the uncertain problem, which finds a robust solution that is immunized against all the possible uncertain scenarios, is then given by the minimax model problem of the form $(P)$,
\begin{equation*}
  \inf\limits_{x\in\mathbb{R}^n} \left\{ \max\limits_{v_0\in \mathcal V_0}  f_0(x,v_0)  :   f_i(x,v_i^j) \leq 0,\ \forall j = 1,\ldots,s_i, \forall i=1,\ldots,k \right\} ,
\end{equation*}
where the uncertain constraints are enforced for every possible value of the parameter $v_i^j$ within their uncertainty sets $\mathcal V_i$.

In the case of standard convex polynomial programming problem where $r=1$ and the functions involved in our model problem $(P)$ are convex polynomials, it is known that there is no duality gap between $(P)$ and its Lagrangian dual \cite{Auslender}. However, the Lagrangian dual, in general, may not easily be solvable. Recent research has shown that whenever $r=1$ and the functions involved in $(P)$ are \emph{SOS-convex polynomials} (see Definition 2.1), the problem $(P)$ enjoys no duality gap between $(P)$ and its dual problem which is representable as a semidefinite programming problem (SDP). Such a duality result is of great interest in optimization because SDP's can efficiently be solved by interior-point methods and so the optimal value of the original model $(P)$ can be found by solving its dual problem \cite{jeya-li-siam}. The new class of SOS-convex polynomials from algebraic geometry \cite{HeNie10,Lasserre} is an important subclass of convex polynomials and it includes convex quadratic functions and separable convex polynomials. The SOS-convexity of polynomials can numerically be checked by solving semidefinite programming problems whereas deciding convexity of polynomials is generally very hard \cite{Parrilo,Parrilo1}.

This raises the very basic issue of which \emph{convex minimax programming problems} can be presented with zero duality gap where the duals can be represented as semidefinite linear programming problems. In this paper we  address  this issue by way of examining minimax programming problems $(P)$ with SOS-convex polynomials. We make the following contributions to minimax optimization.
\smallskip

I. Without any qualifications, we establish dual characterizations of non-negativity of max functions of convex polynomials over a system of convex polynomial inequalities and then derive sum-of-squares-polynomial representations of non-negativity of max functions of SOS-convex polynomials over a system of SOS-convex polynomial inequalities.
\smallskip

II. Using the sum-of-squares-polynomial representations, we introduce a dual program for $(P)$, which is representable as a semidefinite linear programming problem, and show that there is no duality gap between $(P)$ and its dual
whenever the functions $p_j$'s and $g_i$'s are SOS-convex polynomials. Under a constraint qualification, we prove that strong duality holds between $(P)$ and its dual problem. As an application, we prove that the value of a robust convex programming problem under polytopic data uncertainty is equal to its SDP dual program. The significance of our duality theorems is that the value of our model problem $(P)$ can easily be found by solving its SDP dual problem.
\smallskip

III. Under a constraint qualification, we establish that strong duality continues to hold for SOS-convex minimax fractional programming problems with their corresponding SDP duals, including minimax linear fractional programming problems for which the SDP dual problems reduce to linear programming problems.
\smallskip

The outline of the paper is as follows. Section \ref{SEC2} provides dual characterizations and representations of non-negativity of max functions of convex polynomials as well as SOS-convex polynomials over a system of inequalities. Section \ref{SEC3} presents zero duality gaps and strong duality results for our model problem $(P)$. Section \ref{SEC4} gives applications of our duality results to classes of robust convex optimization problems and minimax fractional programming problems. Appendix provides basic re-formulation of our dual problem as semidefinite linear programming problem.


\section{Dual Characterizations and Representations of Non-negativity}
\label{SEC2}


In this Section, we present dual characterizations of solvability of inequality systems involving convex as well as SOS-convex polynomials. Firstly, we shall recall a few basic definitions and results which will be needed later in the sequel. We say that a real polynomial $f$ is \emph{sum of squares} \cite{Laurent_survey} if there exist real polynomials $f_j$, $j=1,\ldots,s$, such that $f=\sum_{j=1}^{s}{f_j^2}$. The set of all sum of squares real polynomials is denoted by $\Sigma^2$, whereas the set consisting of all sum of squares real polynomials with degree at most $d$ is denoted by $\Sigma^2_d$.  Similarly, we say a matrix polynomial $F \in \mathbb{R}[x]^{n \times n}$ is a SOS-matrix polynomial if $F(x)=H(x)H(x)^T$ where $H(x) \in \mathbb{R}[x]^{n \times s}$ is a matrix polynomial for some $s \in \mathbb{N}$. We now introduce the definition of SOS-convex polynomial.

\begin{definition}[{\cite{Parrilo,HeNie10}}]
A real polynomial $f$ on $\mathbb{R}^n$ is called \textit{SOS-convex} if the Hessian matrix function $x\mapsto \nabla^2 f(x)$ is a SOS-matrix polynomial.
\end{definition}

Clearly, a SOS-convex polynomial is convex. However, the converse is not true. Thus, there exists a convex polynomial which is not SOS-convex \cite{Parrilo}. It is known that any convex quadratic function and any convex separable polynomial is a SOS-convex polynomial. Moreover, a SOS-convex polynomial can be non-quadratic and non-separable. For instance, $f(x)=x_1^8+x_1^2+x_1x_2+x_2^2$ is a SOS-convex polynomial (see \cite{Nie0}) which is non-quadratic and non-separable.

The following basic known results on convex polynomials play key roles throughout the paper.
\begin{lemma}[{\cite[Lemma 8]{HeNie10}}]
\label{polysos}
Let $f$ be a SOS-convex polynomial. If $f(u)=0$ and $\nabla f(u)=0$ for some $u\in\mathbb{R}^n$, then $f$ is a sum of squares polynomial.
\end{lemma}


\begin{lemma}[{\cite[Theorem 3]{BeKl02}}]
\label{minattain}
Let $f_0,f_1,\ldots,f_m$ be convex polynomials on $\mathbb{R}^n$. Suppose that $\inf_{x\in C}f_0(x)>-\infty$ where $C:=\{x \in \mathbb{R}^n : f_i(x) \leq 0, i\in\mathbb{N}_m \} \neq \emptyset$.  Then, $\argmin_{x\in C} f_0(x) \neq \emptyset$.
\end{lemma}


\begin{corollary}
\label{polyissos}
Any nonnegative SOS-convex polynomial on $\mathbb{R}^n$ is a sum of squares polynomial. 
\end{corollary}

\begin{proof}
Let $f$ be a nonnegative SOS-convex polynomial on $\mathbb{R}^n$. In virtue of Lemma \ref{minattain}, we know that $\min_{x\in\mathbb{R}^n} f(x) = f(x^*)$ for some $x^*\in\mathbb{R}^n$. Therefore, $h := f - f(x^*) $ is a nonnegative SOS-convex polynomial such that $h(x^*) = 0$ and $\nabla h(x^*)=0$. By applying Lemma \ref{polysos} we get that $h$ is a sum of squares polynomial, so
$ f - f(x^*) = \sigma $
for some $\sigma \in \Sigma^2$. Therefore, $f = \sigma + f(x^*) $ is a sum of squares polynomial since $f(x^*)\geq 0$.
\end{proof}


Let $\Delta$ be the simplex in $\mathbb{R}^r$, that is, $\Delta := \left\{\delta\in \mathbb{R}^r_+ : \sum_{j=1}^r \delta_j = 1 \right\}$.


\begin{theorem}[{Dual characterization of non-negativity}]
\label{qfreeconvexeq}
Let $p_j$ and $g_i$ be convex polynomials for all $j\in \mathbb{N}_r$ and $i\in \mathbb{N}_m$, with $\mathcal F:= \{ x\in\mathbb{R}^n : g_i(x) \leq 0, i\in \mathbb{N}_m\} \neq \emptyset$. Then, the following statements are equivalent:
\begin{enumerate}
\item[\rm($i$)] $g_i(x) \leq 0, \, i\in \mathbb{N}_m \Rightarrow \max\limits_{j\in \mathbb{N}_r} p_j(x) \geq 0 $.
\item[\rm($ii$)] $(\forall \varepsilon>0)$ $(\exists\,\bar{\delta}\in\Delta, \bar{\lambda}\in\mathbb{R}^m_+)$ \ $\sum\limits_{j=1}^{r}{\bar{\delta}_j p_j} + \sum\limits_{i=1}^{m}{\bar{\lambda}_i g_i} + \varepsilon > 0 $.
\end{enumerate}
\end{theorem}

\begin{proof}
$(ii) \Rightarrow (i)$ Suppose that for each $\varepsilon>0$, there exist $\bar{\delta}\in\Delta$ and $\bar{\lambda}\in\mathbb{R}^m_+$ such that $\sum_{j=1}^{r}{\bar{\delta}_j p_j} + \sum_{i=1}^{m}{\bar{\lambda}_i g_i} + \varepsilon > 0 $. Then, for any $x\in\mathcal F$ we have
$$ \max\limits_{\delta\in\Delta} \sum\limits_{j=1}^{r}{\delta_j p_j(x)}  + \varepsilon \geq \sum\limits_{j=1}^{r}{\bar{\delta}_j p_j(x)} + \varepsilon \geq \sum\limits_{j=1}^{r}{\bar{\delta}_j p_j(x)} + \sum\limits_{i=1}^{m}{\bar{\lambda}_i g_i(x)} + \varepsilon > 0. $$
Letting $\varepsilon\rightarrow 0$, we see that  $ \max\limits_{j\in \mathbb{N}_r} p_j(x) = \max\limits_{\delta\in\Delta} \sum\limits_{j=1}^{r}{\delta_j p_j} \geq 0$ for all $x\in\mathcal F$.
\smallskip

$(i) \Rightarrow (ii)$ Assume that $(i)$ holds. Let $\varepsilon>0$ be arbitrary and let $f_j:=p_j+\varepsilon$ for all $j\in \mathbb{N}_r$. Then, one has
\begin{equation*}
  \max\limits_{j\in \mathbb{N}_r} f_j(x)  =  \max\limits_{j\in \mathbb{N}_r} \{p_j(x)\} + \varepsilon  > 0  \quad \forall x\in\mathcal F.
\end{equation*}

Now, we will show that the set
\begin{equation*}
G:=\left\{z=\left(\underline{z},\overline{z}\right)\in \mathbb{R}^{r+m} :  \exists\,x \in\mathbb{R}^n \text{ such that } f_j(x) \leq \underline{z}_j, j\in\mathbb{N}_r,  g_i(x) \leq \overline{z}_i, i\in\mathbb{N}_m \right\}
\end{equation*}
is a closed and convex set. As $f_j$ and $g_i$ are all convex polynomials, then $G$ is clearly a convex set. To see that it is closed, let $\{z^k\}_{k\in\mathbb{N}} \subset G$ be such that $\{z^k\} \rightarrow z^*$ as $k \rightarrow \infty$. Then, for each $k\in\mathbb{N}$, there exists $x^k \in \mathbb{R}^n$ such that $f_j(x^k) \leq \underline{z}^k_j$ and $g_i(x^k) \leq \overline{z}^k_i$, for all $j\in\mathbb{N}_r$ and $i\in\mathbb{N}_m$. Now, consider the convex optimization problem
\begin{equation*}
 \begin{array}{ccl}
  (\bar{P})  &  \min\limits_{x\in\mathbb{R}^n,u\in\mathbb{R}^{r+m}} &  \left\|u-z^*\right\|^2   \\
		            & \text{s.t.}         &     f_j(x)-\underline{u}_j \leq 0, j\in\mathbb{N}_r,    \\
								&                     &     g_i(x)-\overline{u}_i \leq 0, i\in\mathbb{N}_m.
	\end{array}
\end{equation*}
Obviously, $0 \leq \inf(\bar{P}) \leq \left\|z^k-z^*\right\|^2 $ for all $k\in\mathbb{N}$. Since $\left\|z^k-z^*\right\|^2 \rightarrow 0$ as $k \rightarrow \infty$, we get $\inf(\bar{P}) = 0$. Moreover, Lemma \ref{minattain} implies that $\inf(\bar{P})$ is attained, and so, there exists $x^* \in \mathbb{R}^n$ such that $f_j(x^*) \leq \underline{z}^*_j$, $j\in\mathbb{N}_r$, and $g_i(x^*) \leq \overline{z}^*_i$, $i\in\mathbb{N}_m$. So $z^* \in G$, and consequently, $G$ is closed.
\smallskip

Since $\max_{j\in \mathbb{N}_r} f_j(x) > 0$ for all $x\in\mathcal F$,  $0\notin G$. Hence, by the strict separation theorem \cite[Theorem 1.1.5]{Za2002}, there exist $v = \left(\underline{v},\overline{v}\right)\in\mathbb{R}^{r+m} \backslash \{0\}$, $\alpha \in \mathbb{R}$ and $\xi>0$ such that
\begin{equation*}
0 = v^T 0 \leq \alpha < \alpha + \xi \leq \underline{v}^T \underline{z} + \overline{v}^T \overline{z}
\end{equation*}
for all $z \in G$. Since $G + \left(\mathbb{R}^r_+ \times \mathbb{R}^m_+\right) \subset G$,  $\underline{v}_j \geq 0$ and $\overline{v}_i \geq 0$, for all $j\in\mathbb{N}_r$ and $i\in\mathbb{N}_m$. Observe that, for each $x \in \mathbb{R}^n$, $(f_1(x),\ldots,f_r(x),g_1(x),\ldots,g_m(x)) \in G$. So, for each $x \in \mathbb{R}^n$,
\begin{equation}
\label{equ:1}
 \sum_{j=1}^{r} \underline{v}_j f_j(x) + \sum_{i=1}^m \overline{v}_i g_i(x) \geq \alpha+\xi \geq \xi > 0 .
\end{equation}
Now, we claim $\underline{v}\in\mathbb{R}^r_+ \backslash\{0\}$. Otherwise,  $\underline{v} = 0$, then we get from \eqref{equ:1} that $\sum_{i=1}^m \overline{v}_i g_i(\bar{x}) > 0$ for any $\bar{x}\in\mathcal F$ (recall that $\mathcal F$ is nonempty). Since $g_i(\bar{x}) \leq 0$ and $\overline{v}_i \geq 0$ for all $i\in\mathbb{N}_m$,  $\sum_{i=1}^m \overline{v}_i g_i(\bar{x}) \leq 0$, which is a contradiction. So, $\kappa :=\sum_{j=1}^{r} \underline{v}_j > 0$. Therefore, \eqref{equ:1} implies that
\begin{equation*}
 \sum\limits_{j=1}^{r}{\bar{\delta}_j f_j(x)} + \sum\limits_{i=1}^{m}{\bar{\lambda}_i g_i(x)} \geq \bar{\xi} > 0
\end{equation*}
for all $x \in \mathbb{R}^n$, where $\bar{\delta}_j:= \kappa^{-1}\underline{v}_j \geq 0$ for all $j\in\mathbb{N}_r$, $\bar{\lambda}_i:=\kappa^{-1}\overline{v}_i \geq 0 $ for all $i\in\mathbb{N}_m$, and $ \bar{\xi}:= \kappa^{-1} \xi > 0$. Since $\sum_{j=1}^{r}\bar{\delta}_j = 1$, we can write
\begin{equation*}
 \sum\limits_{j=1}^{r}{\bar{\delta}_j p_j} + \sum\limits_{i=1}^{m}{\bar{\lambda}_i g_i} + \varepsilon > 0.
\end{equation*}
Thus, the conclusion follows.
\end{proof}


Let $d$ be the smallest even number such that $d\geq \max\{ \max\limits_{j\in \mathbb{N}_r}\deg p_j,  \max\limits_{i\in \mathbb{N}_m}\deg g_i \}$.


\begin{theorem}[{SOS-Convexity \& representation of non-negativity}]
\label{qfreesosconvexeq}
Let $p_j$ and $g_i$ be SOS-convex polynomials for all $j\in \mathbb{N}_r$ and $i\in \mathbb{N}_m$, with $\mathcal F:= \{ x\in\mathbb{R}^n : g_i(x) \leq 0, i\in \mathbb{N}_m\} \neq \emptyset$. Then, the following statements are equivalent:
\begin{enumerate}
\item[\rm($i$)] $g_i(x) \leq 0, \, i\in \mathbb{N}_m \Rightarrow \max\limits_{j\in \mathbb{N}_r} p_j(x) \geq 0 $.
\item[\rm($ii$)] $(\forall \varepsilon>0)$ $(\exists\,\bar{\delta}\in\Delta, \bar{\lambda}\in\mathbb{R}^m_+,\bar{\sigma}\in\Sigma^2_d)$  $\sum\limits_{j=1}^{r}{\bar{\delta}_j p_j}  + \sum\limits_{i=1}^{m}{\bar{\lambda}_i g_i} + \varepsilon = \bar{\sigma}  $.
\end{enumerate}
\end{theorem}

\begin{proof}
$(ii) \Rightarrow (i)$ Suppose that for each $\varepsilon>0$, there exist $\bar{\delta}\in\Delta$, $\bar{\lambda}\in\mathbb{R}^m_+$ and $\bar{\sigma}\in\Sigma^2_d$ such that $\sum_{j=1}^{r}{\bar{\delta}_j p_j}   + \sum_{i=1}^{m}{\bar{\lambda}_i g_i}  + \varepsilon = \bar{\sigma}$. Then, for any $x\in\mathcal F$ we have
$$ \max\limits_{\delta\in\Delta} \sum\limits_{j=1}^{r}{\delta_j p_j(x)} + \varepsilon \geq \sum\limits_{j=1}^{r}{\bar{\delta}_j p_j(x)} + \varepsilon = \bar{\sigma} - \sum\limits_{i=1}^{m}{\bar{\lambda}_i g_i(x)} \geq 0. $$
Letting $\varepsilon\rightarrow 0$, we see that  $\max\limits_{j\in \mathbb{N}_r} p_j(x) \geq 0$ for all $x\in\mathcal F$.
\smallskip

$(i) \Rightarrow (ii)$ Assume that $(i)$ holds and let $\varepsilon>0$ arbitrary. Then, by Theorem \ref{qfreeconvexeq}, there exist $\bar{\delta}\in\Delta$ and $\bar{\lambda}\in\mathbb{R}^m_+$ such that
$$L:=\sum\limits_{j=1}^{r}{\bar{\delta}_j p_j} + \sum\limits_{i=1}^{m}{\bar{\lambda}_i g_i} + \varepsilon > 0 .$$
Since $p_j$ and $g_i$ are all SOS-convex polynomials, then $L$ is a (nonnegative) SOS-convex polynomial too. Hence, Corollary \ref{polyissos} ensures that $L$ is a sum of squares polynomial (of degree at most $d$), that is, there exist $\bar{\sigma}\in\Sigma^2_d$ such that
$$ \sum\limits_{j=1}^{r}{\bar{\delta}_j p_j} + \sum\limits_{i=1}^{m}{\bar{\lambda}_i g_i} + \varepsilon  = \bar{\sigma} .$$
Thus, the conclusion follows.
\end{proof}


\section{Duality for Minimax Programs with SOS-convex Polynomials}
\label{SEC3}

In this Section we introduce the dual problem for our minimax model problem and establish duality theorems whenever the functions involved are SOS-convex polynomials.

Consider the minimax programming problem
\begin{equation}
\label{primal}
\begin{array}{ccl}
(P)      &   \inf\limits_{x\in\mathbb{R}^n}  &  \max\limits_{j\in \mathbb{N}_r}\,p_j(x)\\
	       &  \text{s.t.}  &  g_i(x) \leq 0, \ i\in \mathbb{N}_m,
\end{array}
\end{equation}
and its associated dual problem
\begin{equation}
\label{dual}
\begin{array}{ccl}
	(D)     &  \sup    &   \mu \\   
	        &   \text{s.t.}   &   \sum\limits_{j=1}^{r}{\delta_j p_j} + \sum\limits_{i=1}^{m}{\lambda_i g_i} -\mu  \in \Sigma^2_d  \\
			    &          &   \delta\in\Delta, \lambda\in\mathbb{R}^m_+, \mu\in\mathbb{R},
\end{array}
\end{equation}
where $p_j$ and $g_i$ are real polynomials on $\mathbb{R}^n$ for all $j\in \mathbb{N}_r$ and $i\in \mathbb{N}_m$ and $d$ is the smallest even number such that $d\geq \max\{ \max\limits_{j\in \mathbb{N}_r}\deg p_j,  \max\limits_{i\in \mathbb{N}_m}\deg g_i \}$.

It is well known that optimization problems of the form $(D)$ can equivalently be re-formula\-ted as semidefinite programming problem \cite{Lasserre}. See Appendix for details. For instance, consider the quadratic optimization problem $(P^{cq})$ where $p_j$ and $g_i$ are all quadratic functions, that is, $p_j(x) = x^T A_j x + a_j^T x + \alpha_j$ and $g_i(x) = x^T C_i x + c_i^T x + \gamma_i$ for all $x\in\mathbb{R}^n$, with $A_j, C_i \in \mathbb{S}^{n}$, the space of all symmetric $(n\times n)$ matrices, $a_j, c_i \in \mathbb{R}^n$ and $\alpha_j,\gamma_i \in \mathbb{R}$ for all $j\in \mathbb{N}_r$ and $i\in \mathbb{N}_m$, that is
\begin{equation}
\label{eq:0}
\begin{array}{ccl}
   (P^{cq})   &   \inf\limits_{x\in\mathbb{R}^n}  &  \max\limits_{j\in \mathbb{N}_r}\,x^T A_j x + a_j^T x + \alpha_j \\
	            &  \text{s.t.}  &  x^T C_i x + c_i^T x + \gamma_i \leq 0,\ \ i\in \mathbb{N}_m.
\end{array}
\end{equation}
In this case, the sum of squares constraint in its associated dual problem $\sum_{j=1}^{r}{\delta_j p_j} + \sum_{i=1}^{m}{\lambda_i g_i} -\mu  \in \Sigma^2_2$ is equivalent to the inequality $\sum_{j=1}^{r}{\delta_j p_j} + \sum_{i=1}^{m}{\lambda_i g_i} -\mu  \geq 0$. This, in turn (see \cite[p. 163]{bental-nemirovski}), is equivalent to
\begin{equation*}
\begin{pmatrix}
\sum\limits_{j=1}^r\delta_j \alpha_j + \sum\limits_{i=1}^m\lambda_i\gamma_i  -\mu    &     \frac{1}{2} (\sum\limits_{j=1}^r \delta_j a_j^T  + \sum\limits_{i=1}^m\lambda_i c_i^T )  \\
\frac{1}{2} (\sum\limits_{j=1}^r \delta_j a_j  + \sum\limits_{i=1}^m\lambda_i c_i  )  &     \sum\limits_{j=1}^r \delta_j A_j + \sum\limits_{i=1}^m\lambda_i C_i
\end{pmatrix} \succeq 0.
\end{equation*}
Therefore, the dual problem of $(P^{cq})$ becomes
\begin{equation}
\label{eq:1}
\begin{array}{ccl}
	(D^{cq})   &  \sup   &   \mu    \\   
	        &   \text{s.t.}   &   \sum\limits_{j=1}^r\delta_j \begin{pmatrix} 2\alpha_j & a_j^T \\ a_j & 2A_j \end{pmatrix}  + \sum\limits_{i=1}^m \lambda_i \begin{pmatrix}  2\gamma_i & c_i^T  \\  c_i & 2C_i\end{pmatrix} - \mu \begin{pmatrix} 2 & 0 \\ 0 & 0 \end{pmatrix} \succeq 0,      \\
			    &          &   \delta\in\Delta, \lambda\in\mathbb{R}^m_+, \mu\in\mathbb{R},
\end{array}
\end{equation}
which is clearly a semidefinite programming problem.


\begin{lemma}
\label{infsup}
Let $p_j$ and $g_i$ be convex polynomials for all $j\in \mathbb{N}_r$ and $i\in \mathbb{N}_m$, with $\mathcal F:= \{ x\in\mathbb{R}^n : g_i(x) \leq 0, i\in \mathbb{N}_m\} \neq \emptyset$. Then,
\begin{equation} \label{zero:gap:1}
\begin{array}{ccc}
	\inf(P)  &  = & \sup\limits_{\delta\in\Delta , \lambda\in\mathbb{R}^m_+}  \inf\limits_{x\in\mathbb{R}^{n}}  \left\{ \sum\limits_{j=1}^{r}{\delta_j p_j(x)} + \sum\limits_{i=1}^{m}{\lambda_i g_i(x)} \right\} .
\end{array}
\end{equation}
\end{lemma}

\begin{proof} Note that, for any $\bar{x}\in\mathcal F$, $\bar{\delta}\in\Delta$ and $\bar{\lambda}\in\mathbb{R}^m_+$, one has
$$ \max\limits_{j\in \mathbb{N}_r} p_j(\bar{x}) \geq \sum\limits_{j=1}^r \bar{\delta}_j p_j(\bar{x}) \geq \sum\limits_{j=1}^r \bar{\delta}_j p_j(\bar{x}) + \sum\limits_{i=1}^m \bar{\lambda}_i g_i(\bar{x}) \geq \inf\limits_{x\in\mathbb{R}^{n}} \left\{ \sum\limits_{j=1}^r \bar{\delta}_j p_j(x) + \sum\limits_{i=1}^m \bar{\lambda}_i g_i(x) \right\}. $$
Therefore, $\inf(P)  \geq \sup_{\delta\in\Delta , \lambda\in\mathbb{R}^m_+}  \inf_{x\in\mathbb{R}^{n}}  \{ \sum_{j=1}^{r}{\delta_j p_j(x)} + \sum_{i=1}^{m}{\lambda_i g_i(x)} \} $.
\smallskip

To see the reverse inequality, we may assume without loss of generality that $\inf(P) > -\infty$, otherwise the conclusion follows immediately. Since $\mathcal F\neq \emptyset$, we have $\mu^*:=\inf(P) \in \mathbb{R}$. Then, for $\varepsilon > 0$ arbitrary, as $\max_{j\in \mathbb{N}_r} \{ p_j(x) -\mu^* \} \geq 0$ for all $x\in\mathcal F$, by Theorem \ref{qfreeconvexeq} we get that there exist $\bar{\delta}\in\Delta$ and $\bar{\lambda}\in\mathbb{R}^m_+$ such that $ \sum_{j=1}^{r}{\bar{\delta}_j p_j}  + \sum_{i=1}^{m}{\bar{\lambda}_i g_i}  > \mu^* - \varepsilon $. Consequently,
$$  \sup\limits_{\delta\in\Delta , \lambda\in\mathbb{R}^m_+}  \inf\limits_{x\in\mathbb{R}^{n}}  \left\{ \sum\limits_{j=1}^{r}{\delta_j p_j(x)} + \sum\limits_{i=1}^{m}{\lambda_i g_i(x)} \right\}  \geq  \mu^* - \varepsilon.$$
Since the above inequality holds for any $\varepsilon>0$, passing to the limit we obtain the desired inequality, which concludes the proof.
\end{proof}


As a consequence of Lemma \ref{infsup}, we derive the following zero-duality gap result for $(P)$.

\begin{theorem}[{Zero duality gap}]
\label{strong:01}
Let $p_j$ and $g_i$ be SOS-convex polynomials for all $j\in \mathbb{N}_r$ and $i\in \mathbb{N}_m$, with $\mathcal F:= \{ x\in\mathbb{R}^n : g_i(x) \leq 0, i\in \mathbb{N}_m\} \neq \emptyset$. Then,
$$\inf(P) = \sup(D).$$
\end{theorem}

\begin{proof} For any $\bar{x}\in\mathcal F$ and any $\bar{\delta}\in\Delta$, $\bar{\lambda}\in\mathbb{R}^m_+$ and $\bar{\mu}\in\mathbb{R}$ such that $\sum_{j=1}^{r}{\bar{\delta}_j p_j} + \sum_{i=1}^{m}{\bar{\lambda}_i g_i} -\bar{\mu} = \bar{\sigma} \in \Sigma^2_d$, one has
$$ \sum\limits_{j=1}^r \bar{\delta}_j \left(p_j(\bar{x}) - \bar{\mu} \right) = \sum\limits_{j=1}^r \bar{\delta}_j p_j(\bar{x}) - \bar{\mu}  =  \bar{\sigma}(\bar{x}) - \sum\limits_{i=1}^m \bar{\lambda}_i g_i(\bar{x})  \geq  0.$$
Then, there exists $j_0\in \mathbb{N}_r$ such that $p_{j_0}(\bar{x})-\bar{\mu} \geq 0$, and so, $ \bar{\mu} \leq \max\limits_{j\in \mathbb{N}_r} p_j(\bar{x}) $. Thus, $\sup(D) \leq \inf(P)$.
\smallskip

To see the reverse inequality, we may assume without loss of generality that $\inf(P)>-\infty$, otherwise the conclusion follows immediately. Since $\mathcal F\neq \emptyset$, we have $\mu^*:=\inf(P) \in \mathbb{R}$. Then, as a consequence of Lemma \ref{infsup}, for $\varepsilon > 0$ arbitrary we have
$$ \sup\limits_{\delta\in\Delta , \lambda\in\mathbb{R}^m_+, \mu\in\mathbb{R}} \left\{ \mu : \sum\limits_{j=1}^{r}{\delta_j p_j} + \sum\limits_{i=1}^{m}{\lambda_i g_i} -\mu \geq 0 \right\}  \geq   \mu^* - \varepsilon. $$
As $p_j$ and $g_i$ are all SOS-convex polynomials, then $L:=\sum_{j=1}^{r}{\delta_j p_j} + \sum_{i=1}^{m}{\lambda_i g_i} -\mu$ is a SOS-convex polynomial too. So, by Corollary \ref{polyissos}, $L$ is nonnegative if and only if $L\in \Sigma^2_d$. Hence, $\mu^* - \varepsilon \leq \sup(D) $. Since the previous inequality holds for any $\varepsilon>0$, passing to the limit we get $\mu^* \leq \sup(D) $, which concludes the proof.
\end{proof}


We now see that whenever the Slater condition,
$$\left\{x\in\mathbb{R}^n : g_i(x) < 0,  i\in \mathbb{N}_m \right\} \neq \emptyset ,$$
is satisfied strong duality between $(P)$ and $(D)$ holds.


\begin{theorem}[{Strong duality}]
\label{strong:02}
Let $p_j$ and $g_i$ be SOS-convex polynomials for all $j\in \mathbb{N}_r$ and $i\in \mathbb{N}_m$, with $\mathcal F:= \{ x\in\mathbb{R}^n : g_i(x) \leq 0, i\in \mathbb{N}_m\} \neq \emptyset$. If the Slater condition holds, then
$$\inf(P) = \max(D).$$
\end{theorem}

\begin{proof}
Let $f:=\max_{j\in \mathbb{N}_r} p_j$ and $\mu^* := \inf(P) \in \mathbb{R}$. Thus, since the Slater condition is fulfilled, by the usual convex programming duality and the convex-convave minimax theorem, we get
\begin{equation*}
 \mu^* \leq \inf(P) = \inf\limits_{x\in\mathbb{R}^n}\left\{ f(x) : g_i(x)\leq 0, i\in \mathbb{N}_m\right\} = \max\limits_{\lambda\in\mathbb{R}^m_+} \, \inf\limits_{x\in\mathbb{R}^n} \left\{ f(x) + \sum_{i=1}^m \lambda_i g_i(x) \right\} =
\end{equation*}
\begin{equation*}
 = \max\limits_{\lambda\in\mathbb{R}^m_+} \, \inf\limits_{x\in\mathbb{R}^n} \, \max_{\delta\in\Delta} \left\{ \sum_{i=1}^{r}{\delta_j p_j(x)} + \sum_{i=1}^m \lambda_i g_i(x) \right\} = \max\limits_{\lambda\in\mathbb{R}^m_+, \delta\in\Delta} \, \inf\limits_{x\in\mathbb{R}^n} \left\{ \sum_{i=1}^{r}{\delta_j p_j(x)} + \sum_{i=1}^m \lambda_i g_i(x) \right\}.
\end{equation*}
Hence, there exist $\bar{\lambda}\in \mathbb{R}^m_+$ and $\bar{\delta}\in\Delta$ such that
\begin{equation*}
L:=\sum\limits_{j=1}^{r}{\bar{\delta}_j p_j} + \sum_{i=1}^{m}{\bar{\lambda}_i g_i} -\mu^* \geq 0.
\end{equation*}
As $p_j$ and $g_i$ are all SOS-convex polynomials, $L$ is a (nonnegative) SOS-convex polynomial too, and consequently, in virtue of Corollary \ref{polyissos}, $L$ is a sum of squares polynomial (of degree at most $d$). Hence, $(\bar{\delta},\bar{\lambda},\mu^*)$ is a feasible point of $(D)$, so $\mu^* \leq \sup(D) $. Since weak duality always holds, we conclude $\inf(P) = \max(D)$.
\end{proof}


Recall the minimax quadratic programming problem $(P^{cq})$ introduced in \eqref{eq:0} and its dual problem $(D^{cq})$ given in \eqref{eq:1}. Note that the set of all $(n\times n)$ positive semi-definite matrices is denoted by $\mathbb{S}^{n}_{+}$.

\begin{corollary}
\label{convex:quadratic}
Let $A_j, C_i \in \mathbb{S}^{n}_{+}$, $a_j, c_i \in \mathbb{R}^n$, and $\alpha_j,\gamma_i \in \mathbb{R}$ for all $j\in \mathbb{N}_r$ and $i\in \mathbb{N}_m$. If there exists $\bar{x}\in\mathbb{R}^n$ such that $\bar{x}^T C_i \bar{x} + c_i^T \bar{x} + \gamma_i < 0$ for all $i\in \mathbb{N}_m$, then
$$\inf(P^{cq}) = \max(D^{cq}).$$
\end{corollary}

\begin{proof}
As $A_j, C_i \in \mathbb{S}^{n}_{+}$ for all $j\in \mathbb{N}_r$ and $i\in \mathbb{N}_m$, all the quadratic functions involved in $(P^{cq})$ are convex. Hence, since the Slater condition holds and any convex quadratic function is a SOS-convex polynomial, by applying Theorem \ref{strong:02} we get $\inf(P^{cq}) = \max(D^{cq})$.
\end{proof}


\begin{remark}[{Attainment of the optimal value}]
\label{minimax:attain}
For the problem $(P)$ introduced in \eqref{primal}, note that if $f:=\max\limits_{j\in\mathbb{N}_r}\,p_j$ (which is not a polynomial, in general) is bounded from below on the nonempty set $\mathcal F$, then $f$ attains its minimum on $\mathcal F$. In other words, if $\inf(P) \in \mathbb{R}$, then there exists $x^*\in\mathcal F$ such that $f(x^*)=\min(P)$. To see this, let consider the following convex polynomial optimization problem.
\begin{equation*}
\begin{array}{ccl}
	(P_{e})  &  \inf\limits_{(x,z)\in\mathbb{R}^{n}\times\mathbb{R}}  &  z \\
	          &  \text{s.t.} & p_j(x) -z \leq 0,\ \forall j\in\mathbb{N}_r,  \\
						&              & g_i(x) \leq 0,  \ \forall i\in\mathbb{N}_m.
\end{array}
\end{equation*}
Let $\mathcal F_e$ be the (nonempty) feasible set of $(P_e)$. Observe that $x_0\in\mathcal F$ implies $(x_0,z_0) \in \mathcal F_e$ for all $z_0 \geq f(x_0)$, and conversely, $(x_0,z_0) \in\mathcal F_e$ implies $x_0\in\mathcal F$. Moreover, one has $\inf(P)=\inf(P_{e})$. Thus, Lemma \ref{minattain} can be applied to problem $(P_e)$ and then, there exists $(x^*,z^*)\in \mathcal F_e$ such that $ z^* = \min(P_e)$. Since $z^* \leq z$ for all $(x,z)\in \mathcal F_e$ and $(x,f(x))\in \mathcal F_e$ for all $x\in \mathcal F$, then we get
\begin{equation}
\label{att1}
 z^* \leq f(x) \qquad \forall x\in \mathcal F.
\end{equation}
On the other hand, as $(x^*,z^*)\in \mathcal F_e$ we get $x^*\in \mathcal F$ and
\begin{equation}
\label{att2}
  f(x^*) \leq z^*.
\end{equation}
Combining \eqref{att1} and \eqref{att2} we conclude $f(x^*) \leq f(x)$ for all $x\in \mathcal F$, and so, $x^*$ is a minimizer of $(P)$.
\end{remark}


Recall that the subdifferential of the (convex) function $f$ at $x\in\mathbb{R}^n$ is defined to be the set
$$ \partial f(x):= \left\{ v\in\mathbb{R}^n : f(y) \geq f(x) + v^T(y-x),\ \forall y\in\dom f \right\}.$$
For a convex set $C\subset \mathbb{R}^{n}$, the normal cone of $C$ of at $x\in C$ is given by
$$ N_C(x):=\left\{v\in \mathbb{R}^n : v^T(y-x)\leq 0,\ \forall y\in C\right\}.$$
Let $\mathcal F:= \{ x\in\mathbb{R}^n : g_i(x) \leq 0, i\in \mathbb{N}_m\} \neq \emptyset$. We will say that the \emph{normal cone condition} holds for $\mathcal F$ at $x\in \mathcal F$ provided that
$$ N_\mathcal F(x) = \left\{ \sum_{i=1}^m \lambda_i \nabla g_i(x)  : \lambda \in \mathbb{R}^{m}_{+}, \sum_{i=1}^{m}{\lambda_i g_i(x)} = 0 \right\}.$$
It is known that the normal cone condition holds whenever the Slater condition is satisfied.

\begin{theorem}[{Min-max duality}]
\label{strong:03}
Let $p_j$ and $g_i$ be SOS-convex polynomials for all $j\in \mathbb{N}_r$ and $i\in \mathbb{N}_m$, with $\mathcal F:= \{ x\in\mathbb{R}^n : g_i(x) \leq 0, i\in \mathbb{N}_m\} \neq \emptyset$. Let $x^*\in\mathcal F$ be an optimal solution of $(P)$ and assume that the normal cone condition for $\mathcal F$ at $x^*$ holds. Then,
$$\min(P) = \max(D).$$
\end{theorem}

\begin{proof}
Let $f:=\max_{j\in \mathbb{N}_r} p_j$ and $\mu^* := \min(P) \in \mathbb{R}$. If $x^*\in\mathcal F$ is an optimal solution of $(P)$, that is, $f(x^*) = \mu^*$, then by optimality conditions we have $0\in \partial f(x^*) + N_{\mathcal F}(x^*)$. As a consequence of the normal cone condition for $\mathcal F$ at $x^*$ and \cite[Proposition 2.3.12]{Clarke}, we get
$$ 0 = \sum\limits_{j=1}^{r}{\bar{\delta}_j \nabla p_j(x^*)} + \sum\limits_{i=1}^m \bar{\lambda}_i \nabla g_i(x^*) $$
for some $\bar{\lambda}\in\mathbb{R}^m_+$ with $\bar{\lambda}_i g_i(x^*) = 0$ for all $i\in \mathbb{N}_m$, and $\bar{\delta}\in \Delta$ with $\bar{\delta}_j=0$ for those $j\in\mathbb{N}_r$ such that $p_j(x^*) \neq \mu^*$. Note that the polynomial
\begin{equation*}
L :=\sum\limits_{j=1}^{r}{\bar{\delta}_j p_j}  + \sum\limits_{i=1}^m \bar{\lambda}_i g_i - \mu^*
\end{equation*}
satisfies $L(x^*)=0$ and $\nabla L(x^*)=0$. Moreover, $L$ is a SOS-convex polynomial since $p_j$ and $g_i$ are all SOS-convex polynomials. Then, as a consequence of Lemma \ref{polysos}, $L$ is a sum of squares polynomial (of degree at most $d$). Then, $(\bar{\delta},\bar{\lambda},\mu^*)$ is a feasible point of $(D)$, so $\mu^* \leq \sup(D) $. Since weak duality always holds, we conclude $\min(P) = \max(D)$.
\end{proof}


It is worth noting that, in the case where $r=1$, our min-max duality Theorem \ref{strong:03} collapses to the corresponding strong duality Theorem 4.1 shown in \cite{jeya-li-siam}.


The following simple example illustrates the above min-max duality therorem.

\begin{example}
Consider the optimization problem
\begin{equation*}
   (P_1) \quad \min\limits_{x\in\mathbb{R}} \left\{ \max\{ 2x^4-x, 5x^2+x \} : x\geq -2 \right\}.
\end{equation*}
It is easy to check that $x^*=0$ is a minimizer of $(P_1)$ and $\min(P_1)=0$. The corresponding dual problem of $(P_1)$ is
\begin{equation*}
   (D_1) \quad \max\limits_{\delta\geq 0,\lambda \geq 0, \mu\in\mathbb{R}} \left\{ \mu :  \delta(2x^4-x) + (1-\delta)(5x^2+x) - \lambda (x+2) - \mu \in \Sigma^2_4 \right\}
\end{equation*}
As $ x^4 + \frac{5}{2}x^2 \in \Sigma^2_4 $,  $\delta = \frac{1}{2}$, $\lambda=0$ and $\mu =0$ is a feasible point of $(D_1)$. So, $\sup(D_1) \geq 0$. On the other hand, the sum of squares constraint in $(D_1)$ gives us $-2\lambda -\mu \geq 0$. Consequently, $\mu \leq -2\lambda \leq 0$, which implies $\max(D) = 0$.
\end{example}


\section{Applications to Robust Optimization \& Rational Programs}
\label{SEC4}


In this Section, we provide applications of our duality theorems to robust SOS-convex programming problems under data uncertainty and to rational programming problems.

Let us consider the following optimization program with the data uncertainty in the constraints and in the objective function.
\begin{equation*}
\begin{array}{ccl}
	(UP)  &   \inf & f_0(x,v_0) \\
	       &  \text{s.t.} & f_i(x,v_i) \leq 0,\ \forall i=1,\ldots,k,
\end{array}
\end{equation*}
where, for each $i\in\{0\}\cup\mathbb{N}_k$, $v_i$ is an uncertain parameter and $v_i \in \mathcal V_i$ for some $\mathcal V_i \subset \mathbb{R}^{n_i}$. The robust counterpart of $(UP)$, which finds a robust solution to $(UP)$ that is immunized against all the possible uncertain scenarios, is given by
\begin{equation*}
\begin{array}{ccl}
(RP)   &   \inf  &  \sup\limits_{v_0\in \mathcal V_0}  f_0(x,v_0) \\
	       &  \text{s.t.}  &  f_i(x,v_i) \leq 0,\ \forall v_i\in\mathcal V_i, \forall i=1,\ldots,k.
\end{array}
\end{equation*}


\begin{theorem}[{Finite data uncertainty}]
Let $f_i(\cdot,v_i)$ be a SOS-convex polynomial for each $v_i\in \mathcal V_i :=\{v_i^1,\ldots,v_i^{s_i} \}$ and each $i\in\{0\}\cup\mathbb{N}_k$ and let $r:=s_0$. Assume that there exists $\bar{x}\in\mathbb{R}^n$ such that $f_i(\bar{x},v_i^j) < 0$ for all $j\in\mathbb{N}_{s_i}$ and $i\in\mathbb{N}_k$. Then $\inf(RP) = \max(RD)$, where
\begin{equation}
\label{concl:robust}
\begin{array}{ccl}
       (RD) &  \sup & \mu  \\   
            & \text{s.t.}   &  \sum\limits_{l=1}^{r}{\delta_l f_0(\cdot,v_0^l)} + \sum\limits_{i=1}^{k}{\sum\limits_{j=1}^{s_i}{ \lambda_i^j f_i(\cdot,v_i^j) }} - \mu \in \Sigma^2_t    \\
						&               &  \delta \in \Delta, \lambda_i\in\mathbb{R}^{s_i}_{+}\ (\forall i\in\mathbb{N}_k), \mu \in \mathbb{R},
\end{array}
\end{equation}
and $t$ is the smallest even number such that $t\geq \max\{ \max\limits_{l\in \mathbb{N}_{r}}\deg f_0(\cdot,v_0^l), \max\limits_{i\in \mathbb{N}_k}\max\limits_{j\in \mathbb{N}_{s_i}}\deg f_i(\cdot,v_i^j)  \}$.
\end{theorem}

\begin{proof}
It is easy to see that problem $(RP)$ is equivalent to
\begin{equation}
\label{dual:aux:rob}
\begin{array}{ccl}
(RP_e)   &  \inf  &  \max\limits_{j\in \mathbb{N}_{r}}  f_0(x,v_0^j) \\
	       &  \text{s.t.}  &  f_i(x,v_i^j) \leq 0,\ \forall j\in\mathbb{N}_{s_i}, \forall i=1,\ldots,k.
\end{array}
\end{equation}
Since the Slater condition holds, by applying Theorem \ref{strong:02} we get $\inf(RP_e) = \max(RD)$.
\end{proof}


\begin{theorem}[{Polytopic data uncertainty}]
Suppose that, for each $i\in\{0\}\cup\mathbb{N}_k$, $x \mapsto f_i(x,v_i)$ is a SOS-convex polynomial for each $v_i \in \mathcal V_i:=\co\{v_i^1,\ldots,v_i^{s_i} \}$ with $r:=s_0$, and $v_i \mapsto f_i(x,v_i)$ is affine for each $x\in\mathbb{R}^{n}$. Assume there exists $\bar{x}\in\mathbb{R}^n$ such that $f_i(\bar{x},v_i^j) < 0$ for all $j\in\mathbb{N}_{s_i}$ and $i\in\mathbb{N}_k$. Then, $\inf(RP) = \max(RD)$ where the problem $(RD)$ is defined in \eqref{concl:robust}.
\end{theorem}

\begin{proof}
Let $i\in\mathbb{N}_k$. As $f_i(x,\cdot)$ is affine for each $x\in\mathbb{R}^{n}$, then $f_i(x,v_i) \leq 0$ for all $v_i\in\mathcal V_i:=\co\{v_i^1,\ldots,v_i^{s_i}\}$ if and only if $f_i(x,v_i^j) \leq 0$ for all $j\in\mathbb{N}_{s_i}$. Moreover, we see that
$$ \sup\limits_{v_0\in \mathcal V_0}  f_0(x,v_0)  = \max\limits_{j\in \mathbb{N}_{r}} f_0(x,v_0^j) .$$
Hence, problem $(RP)$ is equivalent to $(RP_e)$ introduced in \eqref{dual:aux:rob}. Reasoning as in the proof of the above theorem we conclude $\inf(RP) = \max(RD)$.
\end{proof}

\medskip


Now, consider the following minimax rational programming problem,
\begin{equation*}
\begin{array}{ccl}
(\mathcal P)   &   \inf\limits_{x\in\mathbb{R}^n}  &  \max\limits_{j\in \mathbb{N}_r}\,\frac{p_j(x)}{q(x)}\\
	       &  \text{s.t.}  &  g_i(x) \leq 0,\ \ i\in \mathbb{N}_m.
\end{array}
\end{equation*}
where $p_j$, for $j\in \mathbb{N}_r$, $q$, and $g_i$, for $i\in \mathbb{N}_m$, are real polynomials on $\mathbb{R}^n$, and for each $j\in \mathbb{N}_r$, $p_j(x)\geq 0$ and $q(x) > 0 $ over the feasible set. This is a generalization of problem $(P)$ introduced in \eqref{primal}. For related minimax fractional programs, see \cite{crouzeix1,Lai99}. Minimax fractional programs often appear in resource allocation and planning problems of management science where the objective function in their optimization problems involve ratios such as cost or profit in time, return on capital and earnings per share (see \cite{ibraki1}).


We associate with $(\mathcal P)$ the following SDP dual problem

\begin{equation}
\label{dualfrac}
\begin{array}{ccl}
	(\mathcal D)   &  \sup   &   \mu \\  
	        &   \text{s.t.}   &   \sum\limits_{j=1}^{r}{\delta_j p_j} + \sum\limits_{i=1}^{m}{\lambda_i g_i} -\mu q \in\Sigma^2_d  \\
			    &          &   \delta\in\Delta, \lambda\in\mathbb{R}^m_+, \mu\in\mathbb{R},
\end{array}
\end{equation}
where $d$ is the smallest even number such that $d\geq \max\{ \deg q, \max\limits_{j\in \mathbb{N}_r}\deg p_j,  \max\limits_{i\in \mathbb{N}_m}\deg g_i \}$.


It is worth noting that, in general, problem $(\mathcal P)$ may not attain its optimal value when it is finite, even when $r=1$. To see this, consider the rational programming problem $(\mathcal P_1)$ $\inf_{x\in\mathbb{R}}\left\{ \frac{1}{x} : 1-x \leq 0\right\}$. Obviously, $\inf(\mathcal P_1) = 0$, however, for any feasible point $x$, one has $\frac{1}{x}>0$. Thus, the optimal value of $(\mathcal P_1)$ is not attained.


\begin{theorem}[{Strong duality for minimax rational programs}]
\label{strong:frac}
Let $p_j$, $g_i$ and $-q$ be SOS-convex polynomials for all $j\in \mathbb{N}_r$ and $i\in \mathbb{N}_m$, such that $p_j(x)\geq 0$ and $q(x) > 0 $ for all $x\in \mathcal F:= \{ x\in\mathbb{R}^n : g_i(x) \leq 0, i\in \mathbb{N}_m\} \neq \emptyset$. If the Slater condition holds, then
$$\inf(\mathcal P) = \max(\mathcal D).$$
\end{theorem}

\begin{proof}
Note that for any $\mu\in\mathbb{R}_+$, one has $\inf(\mathcal P)\geq \mu$ if and only if $\inf(P_\mu)\geq 0$, where
\begin{equation}
 \label{aux:problem}
 (P_{\mu}) \quad \inf\limits_{x\in\mathcal F}\,\max\limits_{j\in \mathbb{N}_r}\,\{p_j(x)-\mu q(x)\}.
\end{equation}
By the assumption, $\inf(\mathcal P)$ is finite. So, it follows easily that $\mu^* := \inf(\mathcal P) \in\mathbb{R}_+ $ and then $\inf(P_{\mu^*}) \geq 0$. Since, for each $j\in\mathbb{N}_r$, $p_j - \mu^* q$ is a SOS-convex polynomial and the Slater condition holds, by Theorem \ref{strong:02} we have that $\inf(P_{\mu^*}) = \max(D_{\mu^*})$ where
\begin{equation}
\begin{array}{ccl}
	(D_{\mu^*})   &  \sup   &   \theta \\  
	        &   \text{s.t.}   &   \sum\limits_{j=1}^{r}{\delta_j p_j} + \sum\limits_{i=1}^{m}{\lambda_i g_i} -\mu^* q -\theta \in\Sigma^2_d  \\
			    &          &   \delta\in\Delta, \lambda\in\mathbb{R}^m_+, \theta\in\mathbb{R}.
\end{array}
\label{ax:dl}
\end{equation}
As $\max(D_{\mu^*})=\inf(P_{\mu^*}) \geq 0$, there exist $\bar{\delta}\in\Delta$, $\bar{\lambda}\in \mathbb{R}^m_+$ and $\bar{\theta}\in \mathbb{R}_+$ such that
$$ \sum\limits_{j=1}^{r}{\bar{\delta}_j p_j} + \sum_{i=1}^{m}{\bar{\lambda}_i g_i} -\mu^*q \in (\bar{\theta} + \Sigma^2_d) \subset \Sigma^2_d. $$
Therefore, $(\bar{\delta},\bar{\lambda},\mu^*)$ is a feasible point of $(\mathcal D)$, so $\mu^* \leq \sup(\mathcal D) $. Since weak duality always holds, we conclude $\inf(\mathcal P) = \max(\mathcal D)$.
\end{proof}


Let us consider the particular problem $(\mathcal P^{cq})$ where $p_j$, $q$ and $g_i$ are all quadratic functions, that is, $p_j(x) = x^T A_j x + a_j^T x + \alpha_j$, $q(x) = x^T B x + b^T x + \beta$ and $g_i(x) = x^T C_i x + c_i^T x + \gamma_i$ for all $x\in\mathbb{R}^n$, with $A_j, B, C_i \in \mathbb{S}^{n}$, $a_j, b, c_i \in \mathbb{R}^n$ and $\alpha_j,\beta,\gamma_i \in \mathbb{R}$ for all $j\in \mathbb{N}_r$ and $i\in \mathbb{N}_m$, that is
\begin{equation*}
\begin{array}{ccl}
   (\mathcal P^{cq})   &   \inf\limits_{x\in\mathbb{R}^n}  &  \max\limits_{j\in \mathbb{N}_r}\,\displaystyle\frac{x^T A_j x + a_j^T x + \alpha_j}{x^T B x + b^T x + \beta} \\
	            &  \text{s.t.}  &  x^T C_i x + c_i^T x + \gamma_i \leq 0,\ \ i\in \mathbb{N}_m.
\end{array}
\end{equation*}
Assume that $p_j(x)\geq 0$ and $q(x) > 0 $ over the feasible set.
The dual problem of $(\mathcal P^{cq})$ is given by
\begin{equation*}
\begin{array}{ccl}
	(\mathcal D^{cq})   &  \sup   &   \mu    \\   
	        &   \text{s.t.}   &   \sum\limits_{j=1}^r\delta_j \begin{pmatrix} 2\alpha_j & a_j^T \\ a_j & 2A_j \end{pmatrix}  + \sum\limits_{i=1}^m \lambda_i \begin{pmatrix}  2\gamma_i & c_i^T  \\  c_i & 2C_i\end{pmatrix} - \mu \begin{pmatrix} 2\beta & b^T \\ b & 2B \end{pmatrix} \succeq 0,      \\
										&          &   \delta\in\Delta, \lambda\in\mathbb{R}^m_+, \mu\in\mathbb{R},
\end{array}
\end{equation*}
which is clearly a semidefinite programming problem.


\begin{corollary}
\label{convex:quadratic:2}
Let consider the problem $(\mathcal P^{cq})$ such that $A_j, -B, C_i \in \mathbb{S}^{n}_{+}$ for all $j\in \mathbb{N}_r$ and $i\in \mathbb{N}_m$. If there exists $\bar{x}\in\mathbb{R}^n$ such that $\bar{x}^T C_i \bar{x} + c_i^T \bar{x} + \gamma_i < 0$ for all $i\in \mathbb{N}_m$, then
$$\inf(\mathcal P^{cq}) = \max(\mathcal D^{cq}).$$
\end{corollary}

\begin{proof}Note that the sum of squares constraint in its associated dual problem $\sum_{j=1}^{r}{\delta_j p_j} + \sum_{i=1}^{m}{\lambda_i g_i} -\mu q \in \Sigma^2_2$ is equivalent to the inequality $\sum_{j=1}^{r}{\delta_j p_j} + \sum_{i=1}^{m}{\lambda_i g_i} -\mu q \geq 0$. This is equivalent to
\begin{equation*}
\begin{pmatrix}
\sum\limits_{j=1}^r\delta_j \alpha_j + \sum\limits_{i=1}^m\lambda_i\gamma_i  -\mu\beta    &     \frac{1}{2} (\sum\limits_{j=1}^r \delta_j a_j^T  + \sum\limits_{i=1}^m\lambda_i c_i^T -\mu b^T)  \\
\frac{1}{2} (\sum\limits_{j=1}^r \delta_j a_j  + \sum\limits_{i=1}^m\lambda_i c_i  -\mu b)  &     \sum\limits_{j=1}^r \delta_j A_j + \sum\limits_{i=1}^m\lambda_i C_i -\mu B
\end{pmatrix} \succeq 0.
\end{equation*}
So, our dual problem $(\mathcal D)$ collapses to $(\mathcal D^{cq})$.
Since the Slater condition holds and any convex quadratic function is a SOS-convex polynomial, by applying Theorem \ref{strong:frac} we get $\inf(\mathcal P^{cq}) = \max(\mathcal D^{cq})$.
\end{proof}


\begin{corollary}
Let $p$, $g_i$ and $-q$ be SOS-convex polynomials for all $i\in \mathbb{N}_m$, such that $p(x)\geq 0$ and $q(x) > 0 $ for all $x\in \mathcal F:= \{ x\in\mathbb{R}^n : g_i(x) \leq 0, i\in \mathbb{N}_m\} \neq \emptyset$. If the Slater condition holds, then
$$ \inf\limits_{x\in\mathcal F} \frac{p(x)}{q(x)}  = \max\limits_{\mu\in\mathbb{R},\lambda\in\mathbb{R}^m_+}\left\{ \mu : p + \sum_{i=1}^{m}{\lambda_i g_i} -\mu q \in \Sigma^2_k  \right\}$$
where $k$ is the smallest even number such that $k\geq \max\{ \deg p, \deg q, \max\limits_{i\in \mathbb{N}_m}\deg g_i \}$.
\end{corollary}

\begin{proof}
It is a straightforward consequence of Theorem \ref{strong:frac} when $r=1$.
\end{proof}


Next we show that the non-negativity of the polynomials $p_j$'s can be dropped whenever $q$ is an affine function.

\begin{corollary}
\label{strong:4}
Let $p_j$ and $g_i$ be SOS-convex polynomials for all $j\in \mathbb{N}_r$ and $i\in \mathbb{N}_m$, $b\in\mathbb{R}^n$ and $\beta\in\mathbb{R}$ such that $b^Tx+\beta>0$ for all $x\in \mathcal F:= \{ x\in\mathbb{R}^n : g_i(x) \leq 0, i\in \mathbb{N}_m\} \neq \emptyset$. If the Slater condition holds, then
\begin{equation*}
\inf\limits_{x\in\mathcal F} \, \max\limits_{j\in \mathbb{N}_r} \, \frac{p_j(x)}{b^Tx+\beta} = \max\limits_{\substack{\mu\in\mathbb{R} \\ \delta\in\Delta, \lambda\in\mathbb{R}^m_+}} \left\{ \mu : \sum\limits_{j=1}^{r}{\delta_j p_j(x)} + \sum\limits_{i=1}^{m}{\lambda_i g_i(x)} -\mu (b^Tx + \beta)  \in \Sigma^2_d\right\}.
\end{equation*}
\end{corollary}

\begin{proof}
The proof follows the same line of arguments as the proof of Theorem \ref{strong:frac}, except that, in the case $q(x):=b^Tx+\beta$ for all $x\in\mathbb{R}^{n}$, all polynomials $p_j -\mu^*q $ are SOS-convex without the non-negativity of all $p_j$'s, and therefore, of $\mu^*$.
\end{proof}


For the particular problem

\begin{equation*}
\begin{array}{ccl}
	(\mathcal P^{l})  & \inf\limits_{x\in\mathbb{R}^n} & \max\limits_{j\in \mathbb{N}_r}\,\frac{a_j^T x + \alpha_j}{b^T x + \beta}   \\
	       & \text{s.t.} &  c_i^T x + \gamma_i \leq 0, \ i\in \mathbb{N}_m,
\end{array}
\end{equation*}
the corresponding dual problem can be stated as the following linear programming problem
\begin{eqnarray}
	(\mathcal D^{l})   &  \max    &   \mu   \nonumber \\
	        &   \text{s.t.}   &   \sum\limits_{j=1}^r \delta_j a_j + \sum\limits_{i=1}^m\lambda_i c_i   -\mu b = 0,  \label{const:lin1} \\
					&          &   \sum\limits_{j=1}^r \delta_j \alpha_j  + \sum\limits_{i=1}^m\lambda_i\gamma_i - \mu\beta \geq 0,   \label{const:lin2}  \\
			    &          &   \delta\in\Delta, \lambda\in\mathbb{R}^m_+, \mu\in\mathbb{R}.  \nonumber
\end{eqnarray}

\begin{corollary}
Let $\alpha_j,\beta,\gamma_i \in \mathbb{R}$ and $a_j,b,c_i \in \mathbb{R}^n$ for all $j\in \mathbb{N}_r$ and $i\in \mathbb{N}_m$. Assume that $b^T x + \beta > 0$ for all feasible point $x$ of $\mathcal P^l$. Then,
$$ \inf(\mathcal P^{l}) = \max(\mathcal D^{l}).$$
\end{corollary}

\begin{proof}
By applying Corollary \ref{strong:4}, we get that $\inf(\mathcal P^{l})$ equals to
\begin{equation*}
\max\limits_{\substack{\mu\in\mathbb{R} \\ \delta\in\Delta, \lambda\in\mathbb{R}^m_+}} \left\{  \mu  : \sum\limits_{j=1}^r \delta_j \left( a_j^T x + \alpha_j  \right)  + \sum\limits_{i=1}^m\lambda_i \left(c_i^T x + \gamma_i\right) - \mu \left( b^T x + \beta\right) \in \Sigma^2_2  \right\}
\end{equation*}
Since the sum of squares constraint in the above dual problem is equivalent to \eqref{const:lin1} and \eqref{const:lin2}, we conclude $\inf(\mathcal P^{l}) = \max(\mathcal D^{l})$.
\end{proof}


If a minimizer $x^*$ of $(\mathcal P)$ is known, then the Slater condition in Theorem \ref{strong:frac} can be replaced by a weaker condition in order to derive strong duality between $(\mathcal P)$ and $(\mathcal D)$.

\begin{theorem}
\label{strong:frac:min}
Let $p_j$, $g_i$ and $-q$ be SOS-convex polynomials for all $j\in \mathbb{N}_r$ and $i\in \mathbb{N}_m$, such that $p_j(x)\geq 0$ and $q(x) > 0 $ for all $x\in \mathcal F:= \{ x\in\mathbb{R}^n : g_i(x) \leq 0, i\in \mathbb{N}_m\} \neq \emptyset$. Let $x^*\in\mathcal F$ be an optimal solution of $(\mathcal P)$ and assume that the normal cone condition for $\mathcal F$ at $x^*$ holds. Then,
$\min(\mathcal P) = \max(\mathcal D).$
\end{theorem}

\begin{proof}
Let $\mu^* := \min(\mathcal P) \in \mathbb{R}_{+}$. Note that $(\mathcal P)$ has optimal solution $x^*$ with optimal value $\mu^*$ if and only if $x^*$ is an optimal solution of $(P_{\mu^*})$ with optimal value $0$ (cf. \cite[Lemma 2.3]{Lai99}), where $(P_{\mu^*})$ is stated in \eqref{aux:problem}. Since, for each $j\in\mathbb{N}_r$, $p_j - \mu^* q$ is a SOS-convex polynomial and the normal cone condition for $\mathcal F$ at $x^*$ holds, by Theorem \ref{strong:03} we have that $\min(P_{\mu^*}) = \max(D_{\mu^*})$ where $(D_{\mu^*})$ has been stated in \eqref{ax:dl}. As $\max(D_{\mu^*}) = 0$, there exist $\bar{\delta}\in\Delta$ and $\bar{\lambda}\in \mathbb{R}^m_+$ such that
$$ \sum\limits_{j=1}^{r}{\bar{\delta}_j p_j} + \sum_{i=1}^{m}{\bar{\lambda}_i g_i} -\mu^*q \in \Sigma^2_d. $$
Therefore, $(\bar{\delta},\bar{\lambda},\mu^*)$ is a feasible point of $(\mathcal D)$, so $\mu^* \leq \sup(\mathcal D) $. Since weak duality always holds, we conclude $\min(\mathcal P) = \max(\mathcal D)$.
\end{proof}


\section*{Appendix: SDP Representations of Dual Programs.}
\label{APP}


Finally, for the sake of completeness, we show how our dual problem $(\mathcal D)$ given in \eqref{dualfrac} can be represented by a semidefinite linear programming problem. To this aim, let us recall some basic facts on the relationship between sums of squares polynomials and semidefinite programming problems.

We denote by $\mathbb{S}^{n}$ the space of symmetric $n\times n$ matrices. For any $A,B\in \mathbb{S}^{n}$, we write $A \succeq 0$ if and only if $A$ is positive semidefinite, and $\left\langle A, B\right\rangle$ stands for $\operatorname{trace}(AB)$. Let $\mathbb{S}^{n}_{+}:=\{ A \in \mathbb{S}^{n} : A\succeq 0\}$ be the closed convex cone of positive semidefinite $n\times n$ (symmetric) matrices. The space of all real polynomials on $\mathbb{R}^n$ with degree $d$ is denoted by $\mathbb{R}_d[x_1,\ldots,x_n]$ and its canonical basis is given by 
\begin{equation*}
  y(x) \equiv (x_{\alpha})_{\left|\alpha\right| \leq d} := (1,x_1,x_2,\ldots,x_n,x_1^2,x_1x_2,\ldots,x_2^2,\ldots,x_n^2,\ldots,x_1^{d},\ldots,x_n^d)^T,
\end{equation*}
which has dimension $e(d,n):=\binom{n+d}{d}$, and $\alpha \in \mathbb{N}^n$ is a multi-index such that $\left|\alpha\right|:=\sum_{i=1}^{n}{\alpha_i}$. Let $\mathcal N:=\{ \alpha\in \mathbb{N}^n : \left|\alpha\right| \leq d \}$. Thus, if $f$ is a polynomial on $\mathbb{R}^n$ with degree at most $d$, one has
$$ f(x) = \sum\limits_{\alpha\in \mathcal{N}} f_{\alpha}x_{\alpha}. $$

Assume that $d$ is an even number, and let $k:=d/2$. Then, according to \cite[Proposition 2.1]{Lasserre}, $f$ is a sum of squares polynomial if and only if there exists $Q \in \mathbb{S}_{+}^{e(k,n)}$ such that $f(x)= y(x)^T Q \,y(x)$. By writing $y(x) y(x)^T = \sum_{\alpha\in \mathcal{N}} B_{\alpha}x_{\alpha}$ for appropiate matrices $(B_{\alpha})\subset \mathbb{S}^{e(k,n)}$, one has that $f$ is a sum of squares polynomial if and only if there exists $Q \in \mathbb{S}_{+}^{e(k,n)}$ such that $\left\langle Q,B_{\alpha}\right\rangle = f_{\alpha}$ for all $\alpha \in \mathcal{N}$.
\smallskip

Using the above characterization, we see that our dual problem $(\mathcal D)$ can be equivalently rewritten as the following semidefinite programming problem.
\begin{equation*}
\begin{array}{ccl}
(\mathcal{SD}) & \sup   &   \mu \\
     & \text{s.t.}   &   \sum\limits_{j=1}^{r}{\delta_j (p_j)_{\alpha}} + \sum\limits_{i=1}^{m}{\lambda_i (g_i)_{\alpha}} - \mu q_{\alpha} = \left\langle Q,B_{\alpha}\right\rangle \qquad \forall \alpha \in \mathcal{N}, \\
     &   & \sum\limits_{j=1}^{r}{\delta_j} = 1, \\
     &   &  \delta\in\mathbb{R}^r_+, \lambda\in\mathbb{R}^m_+, \mu\in\mathbb{R}, Q \in \mathbb{S}^{e(k,n)}_{+}.
\end{array}
\end{equation*}

Letting $q_{\alpha}=1$ for $\alpha=(0,\ldots,0)$ and $q_{\alpha}=0$ otherwise, we get the SDP representation for problem $(D)$ in \eqref{dual}.




\end{document}